\documentclass[12pt]{amsart}
\usepackage{geometry}   
\usepackage[colorlinks,citecolor = red, linkcolor=blue,hyperindex]{hyperref}
\usepackage{euscript,eufrak,verbatim, mathrsfs, mathtools, amscd}
\usepackage{yhmath}
\usepackage[psamsfonts]{amssymb}
\usepackage[usenames]{color}
\usepackage{bbm}
\usepackage{graphicx}
 \usepackage{color}
 \usepackage{float}
\usepackage{tikz-cd}

\usepackage{dsfont}

 \usepackage{euscript}
 \usepackage{xypic}
\usepackage{helvet}         
\usepackage{courier}        
\usepackage{type1cm}        
\usepackage{multicol}        
\usepackage[bottom]{footmisc}

\newtheorem{theorem}{Theorem}[section]
\newtheorem*{theorem*}{Theorem B}
\newtheorem{lemma}[theorem]{Lemma}

\newtheorem{proposition}[theorem]{Proposition}
\newtheorem{corollary}[theorem]{Corollary}

\newtheorem{question}[theorem]{Question}
\newtheorem*{definition*}{Definition}

\newtheorem*{observation*}{Observation}

\newtheorem*{assumption*}{Assumption}
\newtheorem*{question*}{Question}

\theoremstyle{definition}

\newtheorem*{remark*}{Remark}

\newcommand{\R}{\mathbb{R}}
\newcommand{\N}{\mathbb{N}}
\newcommand{\Z}{\mathbb{Z}}

\newcommand{\C}{\mathbb{C}}

\newcommand{\PP}{\mathbb{P}}

\newcommand{\Conf}{\mathrm{Conf}}

\newcommand{\supp}{\mathrm{supp}}

\newcommand{\ct}{\mathrm{T}}
\newcommand{\aut}{\mathrm{Aut}}
\newcommand{\rad}{\mathrm{rad}}

\begin{document}

\title[Radial Toeplitz operators and invariant DPPs on trees]{Radial Toeplitz operators and invariant determinantal point processes on Cayley trees}

\author
{Yanqi Qiu}
\thanks{This research is supported by grants NSFC Y7116335K1,  NSFC 11801547 and NSFC 11688101 of National Natural Science Foundation of China.}
\address
{Yanqi Qiu:Institute of Mathematics and Hua Loo-Keng Key Laboratory of Mathematics, AMSS, Chinese Academy of Sciences, Beijing 100190, China.}
\email{yanqi.qiu@amss.ac.cn; yanqi.qiu@hotmail.com}

\begin{abstract}
We give a complete description of bounded radial Toeplitz operators on the Hilbert space associated with  a Cayley tree. As an application, we give a complete classification of a rich family of determinantal point processes on Cayley trees whose correlation kernels  and thus distributions are invariant under the action of the graph automorphisms of the Cayley trees.
\end{abstract}


\maketitle

\setcounter{equation}{0}

\section{Introduction}

\subsection{Radial Toeplitz operators on Cayley trees}
Fix an integer $\kappa \ge 1$  and consider the Cayley tree $\ct_\kappa = (V_\kappa, E_\kappa)$ of order $\kappa$, i.e.,  $\ct_\kappa$ is an infinite undirected connected graph without cycles such that each vertex has exactly $\kappa + 1$ edges. Let $d(\cdot, \cdot)$ denote the graph distance on the set $V_\kappa$ of vertices.   


An edge with end points $x, y \in V_\kappa$ will be denoted by $\wideparen{xy}\in E_\kappa$.  A bijection $\gamma: V_\kappa \rightarrow V_\kappa$ is called a {\it graph automorphism}  of $\ct_\kappa$ provided that $\wideparen{xy} \in E_\kappa$ if and only if $\wideparen{\gamma(x) \gamma(y)} \in E_\kappa$ for any pair $(x, y)$ of distinct vertices of $\ct_\kappa$.  Let $\Gamma = \aut(\ct_\kappa)$ denote the group of the graph automorphisms of $\ct_\kappa$.  Note that any automorphism $\gamma \in \Gamma$ preserves the graph distance  on $V_\kappa$.  

By a kernel on $V_\kappa$, we mean a two-variable function   $K: V_\kappa \times V_\kappa \rightarrow \C$. A kernel $K$  on $V_\kappa$ is called $\Gamma$-equivariant, if 
\[
K(x, y) = K(\gamma(x), \gamma(y)),   \quad \forall x, y \in V_\kappa, \gamma \in \Gamma.
\]
To any function 
$\alpha: \N_0 \rightarrow \C$  defined on the set $\N_0 = \{0, 1, 2, \cdots \}$  is associated a {\it radial Toeplitz kernels} given by 
\begin{align}\label{def-T-phi}
T_\alpha(x, y) = \alpha(d(x, y)), \quad \forall x, y \in V_\kappa. 
\end{align}
Clearly any  $\Gamma$-equivariant kenrel on $V_\kappa$ is of the above radial Toeplitz form.  We denote by $\mathcal{T}_\Gamma(V_\kappa)$ the set of all radial Toeplitz kernels on $V_\kappa$:
\begin{align}\label{def-radial-Toep}
\mathcal{T}_\Gamma(V_\kappa) : = \Big\{K: V_\kappa \times V_\kappa \rightarrow \C \Big| \text{$K = T_\alpha$ for some $\alpha:  \N_0 \rightarrow \C$}\Big\}. 
\end{align}

Let  $B(\ell^2(V_\kappa))$ denote the $C^*$-algebra  consisting  of bounded operators on the standard complex Hilbert space $\ell^2(V_\kappa)$ defined by
\[
\ell^2(V_\kappa): = \Big\{ \sum_{x\in V_\kappa} c_x \delta_x  \Big| \text{$c_x \in \C$ and $\sum_{x} |c_x|^2 <\infty$}\Big\},
\] 
where $\delta_x$ is the Dirac function on $V_\kappa$ at the point $x$.  We will identify any bounded operator $A \in B(\ell^2(V_\kappa))$ with its kernel:
\[
A(x, y) = \langle A \delta_y, \delta_x\rangle, \quad \forall x, y \in V_\kappa,
\]
where $\langle \cdot, \cdot \rangle$ is the standard inner product on $\ell^2(V_\kappa)$. 
Thus if a kernel $K: V_\kappa \times V_\kappa \rightarrow \C$ induces a bounded operator on $\ell^2(V_\kappa)$, then we will use the same notation $K$ for the induced bounded operator. Using this convention, the set
\[
B_\Gamma(\ell^2(V_\kappa)): =  B(\ell^2(V_\kappa)) \cap \mathcal{T}_\Gamma(V_\kappa)
\]
will be viewed both as the subset of bounded operators on $\ell^2(V_\kappa)$ with radial Toeplitz kernels and  as the set of radial Toeplitz kernels inducing bounded operators on $\ell^2(V_\kappa)$.

We are going to   give a complete description of $B_\Gamma(\ell^2(V_\kappa))$. For this purpose, define recursively the Cartier-Dunau polynomials  (see Arnaud \cite{Arnaud94})  by   
\begin{align}\label{cartier-dunau}
P_0(t)   = 1,  \quad P_1(t) = t, \quad t P_n(t) = \frac{\kappa}{ \kappa + 1} P_{n+1} (t) + \frac{1}{\kappa + 1} P_{n-1}(t),  \quad \forall n\ge 1. 
\end{align}
By Favard's theorem,  $(P_n)_{n\ge 0}$ is a sequence of orthogonal polynomials with respect to a unique probability  masure on $\R$. This probability measure is called the orthogonality probability measure of  the sequence $(P_n)_{n\ge 0}$ and is given by the following explicit formula (see the Appendix of this paper): 
\begin{align}\label{pi-kappa}
d\Pi_\kappa (t) =  \frac{\kappa + 1}{2\pi }  \frac{\sqrt{4\kappa (\kappa + 1)^{-2} - t^2}}{1 - t^2}  \mathds{1}_{[- \frac{2 \sqrt{\kappa} }{ \kappa + 1}, \,  \frac{2 \sqrt{\kappa} }{ \kappa + 1} ]} (t) dt.
\end{align}

For simplifying notation, for any $\varphi \in L^1(\Pi_\kappa)$, we set 
\begin{align}\label{def-hat}
\widehat{\varphi}(n): = \int_\R   P_n(t)  \varphi(t) d \Pi_\kappa (t), \quad \forall n \ge 0.
\end{align}
Thus to any $\varphi\in L^1(\Pi_\kappa)$ is associated a radial Toeplitz kernel:
\begin{align}\label{def-T-varphi}
T[\varphi]  :  = T_{\widehat{\varphi}} \in \mathcal{T}_\Gamma(V_\kappa). 
\end{align}

\begin{lemma}\label{lem-alg}
The subset $B_\Gamma(\ell^2(V_\kappa)) \subset B(\ell^2(V_\kappa))$ forms a commutative   sub-$C^*$-algebra.
\end{lemma}

\begin{theorem}\label{thm-alg-iso}
The following mapping
\begin{align}\label{def-iota}
\begin{array}{cccc}
\iota: & L^\infty(\Pi_\kappa) &\rightarrow & B_\Gamma(\ell^2(V_\kappa))
\vspace{2mm}
\\
& \varphi & \mapsto & T[\varphi]
\end{array}
\end{align}
defines an isomorphism between two commutative  $C^*$-algebras.  
\end{theorem}

\begin{remark*}
Since the map \eqref{def-iota} defines an isomorphism between two $C^*$-algebras, for any fixed function $\varphi \in L^\infty(\Pi_\kappa)$, the spectrum of the operator $T[\varphi]$ is given by the essential range of the function $\varphi$. 
\end{remark*}

Thereom \ref{thm-alg-iso}  can be made more precise in the following Theorem \ref{prop-inv}. 

Fix any vertex $o \in V_\kappa$. A vector $\eta  = \sum_{x\in V_\kappa} \eta_x \delta_x \in \ell^2(V_\kappa)$ is called {\it a radial vector}  with respect to the vertex $o$ if 
\[
\eta_x = \eta_y, \quad \text{whenever $d(x, o) = d(y, o)$}.
\] 
Let $\ell^2(V_\kappa)_{\rad (o)} \subset \ell^2(V_\kappa)$ denote the subspace of radial vectors with respect to $o$ and let $P_{\rad(o)}$ denote the orthogonal projection from $\ell^2(V_\kappa)$ onto the subspace $\ell^2(V_\kappa)_{\rad(o)}$. By setting $S_n(o) = \{ x \in V_\kappa | d(x, o) = n\}$, we get a natural orthonormal basis of $\ell^2(V_\kappa)_{\rad(o)}$:
\begin{align}\label{onb-f-n}
 f_n = \frac{\mathds{1}_{S_n(o)}}{\sqrt{\# S_n(o)}} = \frac{1}{\sqrt{\# S_n(o)}}\sum_{x  \in S_n(o)}  \delta_x, \quad \forall n \ge 0.
\end{align}

\begin{lemma}\label{lem-inv}
For any $o\in V_\kappa$, the subspace $\ell^2(V_\kappa)_{\rad(o)}$ is a common invariant  subspace of all operators in $B_\Gamma(\ell^2(V_\kappa))$.
\end{lemma}

\begin{theorem}\label{prop-inv}
For any $o \in V_\kappa$, the following map 
\begin{align}\label{def-res}
\begin{array}{cccc}
j_o: & B_\Gamma(\ell^2(V_\kappa))  &\rightarrow &  B(\ell^2(V_\kappa)_{\rad(o)})
\vspace{2mm}
\\
& T[\varphi] & \mapsto & P_{\rad(o)} T[\varphi] P_{\rad(o)}
\end{array}
\end{align}
defines an isometric embedding of $C^*$-algebras. Furthermore, we have the following commutative diagram simultaneously for all $\varphi \in L^\infty(\Pi_\kappa)$: 
\begin{align}\label{CD-unitary}
\begin{CD}
\ell^2(V_\kappa)_{\rad(o)} @> P_{\rad(o)} T[\varphi] P_{\rad(o)}>>   \ell^2(V_\kappa)_{\rad(o)}
\\
@V{U_o}V{\simeq }V     @V{U_o}V{\simeq}V
\\
 L^2(\Pi_\kappa) @>\qquad M_\varphi \qquad >> L^2(\Pi_\kappa)
\end{CD},
\end{align}
where $M_\varphi$ is the operator of multiplication defined by 
\[
M_\varphi(\psi) = \varphi \psi, \quad \text{for all $\psi \in L^2(\Pi_\kappa)$}
\]
 and $U_o: \ell^2(V_\kappa)_{\rad(o)}\rightarrow L^2(\Pi_\kappa)$ is the unitary operator determined by 
\begin{align}\label{def-U-o}
U_o \left( \frac{\mathds{1}_{S_n(o)}}{\sqrt{\# S_n(o)}} \right) = \frac{P_n}{\| P_n\|_{L^2(\Pi_\kappa)}}, \quad  \forall n \ge 0. 
\end{align}
\end{theorem}

The following result is an immediate consequence of Theorem \ref{prop-inv}. 
\begin{corollary}\label{cor-radial}
For any  operator  $A \in B_\Gamma(\ell^2(V_\kappa))$ and any vertex $o \in V_\kappa$,  we have $\| A\| = \| A P_{\rad(o)}\|$. That is, 
\begin{align}\label{norm-eq-rad}
\sup_{\eta \in \ell^2(V_\kappa) \setminus \{0\}} \frac{\| A  \eta \|}{ \| \eta\|}= \sup_{\xi \in \ell^2(V_\kappa)_{\rad(o)} \setminus \{0\}} \frac{\| A  \xi \|}{ \| \xi\|}.
\end{align}
\end{corollary}

\begin{remark*}
The equality \eqref{norm-eq-rad} seems to be non-trivial even in the case $\kappa =1$: it holds only for infinite radial symmetric vectors. More precisely, 
set $B_n(o): = \{x \in V_\kappa| d(x, o) \le n\}$, consider $\ell^2(B_n(o))$ as a subspace of $\ell^2(V_\kappa)$ and set
\[
\ell^2(B_n(o))_{\rad(o)}  = \ell^2(B_n(o)) \cap  \ell^2(V_\kappa)_{\rad(o)}.
\]
Then for a general bounded operator $A\in B_\Gamma(\ell^2(V_\kappa))$ and a finite $n$, the following equality is in general not true:
\[
\sup_{\eta \in \ell^2(B_n(o)) \setminus \{0\}} \frac{\| A  \eta \|}{ \| \eta\|}= \sup_{\xi \in \ell^2(B_n(o))_{\rad(o)} \setminus \{0\}} \frac{\| A  \xi \|}{ \| \xi\|}.
\]
Indeed, consider  the case $\kappa =1$ and Toeplitz operators on $\ell^2(\Z)$. Let $\alpha: \N_0 \rightarrow \C$ be the function $\alpha  = \delta_0 - b \delta_2$  (here we assume that $0< b < 1$). Then  $T_\alpha$ is given by 
\[
T_\alpha: = \left[
\begin{array}{ccccccc}
\ddots& \ddots& \ddots& \ddots& \ddots&  \ddots& 
\\ 
-b& 0 & 1 & 0 & - b&0 & 0
\\
0 & -b& 0 & 1 & 0 & -b& 0
\\
0& 0 & - b & 0 & 1 & 0 & -b
\\
 & \ddots& \ddots& \ddots& \ddots& \ddots& \ddots
\end{array}
\right].
\]
For $n = 1$, we have $B_1(0) = \{-1, 0, 1\}$. Take any $\eta = \eta_{-1} \delta_{-1} + \eta_0 \delta_0 + \eta_1 \delta_1  \in \ell^2(\{-1, 0, 1\})$, we have
\begin{align*}
\|T_\alpha \eta \|^2 = (2b^2 +1)| \eta_{0}|^2 + (2b^2 +1)| \eta_{-1}|^2 + (2b^2 +1)| \eta_{1}|^2 - 2b \eta_{-1}\overline{\eta_1}- 2b \overline{\eta_{-1}} \eta_1.
\end{align*}
We can obtain the following strict inequality
\[
\sup_{\eta \in \ell^2(\{-1, 0, 1\}) \setminus \{0\}} \frac{\| T\eta\|}{\| \eta\|} = \sqrt{2b^2 + 2 b + 1} > \sqrt{2 b^2 +1} = \sup_{\eta \in \ell^2(\{-1, 0, 1\})\setminus \{0 \} \atop \eta_{-1} = \eta_1 } \frac{\| T\eta\|}{\| \eta\|}. 
\]
\end{remark*}

\bigskip

Theorem \ref{prop-inv} has the following vector-valued extension.  Let $\mathcal{H}$ be any complex Hilbert space and let $\ell^2(V_\kappa; \mathcal{H})$ be the Hilbert space defined by
\[
\ell^2(V_\kappa; \mathcal{H}): = \Big\{ \sum_{x \in V_\kappa}\xi_x \delta_x \Big| \text{$\xi_x\in \mathcal{H}$ and $\sum_{x \in V_\kappa} \| \xi_x\|_\mathcal{H}^2 <\infty$}\Big\}.
\]
Recall that all bounded operators on $\ell^2(V_\kappa; \mathcal{H})$ can be written in block forms $[A(x, y)]_{x, y\in V_\kappa}$ with each block $A(x, y) \in B(\mathcal{H})$. Let $B_\Gamma(\ell^2(V_\kappa; \mathcal{H}))$ denote the space of bounded operators $[A(x, y)]_{x, y\in V_\kappa}$ on $\ell^2(V_\kappa; \mathcal{H})$ such that 
\[
A(x, y) = A(\gamma (x), \gamma (y)), \quad \forall x, y \in V_\kappa, \gamma \in \Gamma. 
\]
Such operators must be of radial Toeplitz block forms: 
\[
A(x, y) = \Theta(d(x, y)) = : T_\Theta(x, y), \quad \forall x, y \in V_\kappa,
\]
where $\Theta: \N_0 \rightarrow B(\mathcal{H})$ is a $B(\mathcal{H})$-valued function. 

Fix any vertex $o \in V_\kappa$. Similar to the definition of $\ell^2(V_\kappa)_{\rad (o)}$, we define  the subspace $\ell^2(V_\kappa; \mathcal{H})_{\rad (o)} \subset \ell^2(V_\kappa; \mathcal{H})$ of radial vectors with respect to the vertex $o$. 

\begin{proposition}
For any  vertex $o\in V_\kappa$, the subspace $\ell^2(V_\kappa; \mathcal{H})_{\rad(o)}$ is a common invariant  subspace of all operators in $B_\Gamma(\ell^2(V_\kappa; \mathcal{H}))$. Moreover, we have the following commutative diagram simultaneously for all $\Phi \in L^\infty(\Pi_\kappa; B(\mathcal{H}))$: 
\begin{align*}
\begin{CD}
\ell^2(V_\kappa; \mathcal{H})_{\rad(o)} @> (P_{\rad(o)} \otimes I_H) T[\Phi] ( P_{\rad(o)} \otimes I_H)>>   \ell^2(V_\kappa; \mathcal{H})_{\rad(o)}
\\
@V{U_o \otimes I_H}V{\simeq }V     @V{U_o \otimes I_H}V{\simeq}V
\\
 L^2(\Pi_\kappa; \mathcal{H}) @>\qquad  \qquad M_\Phi \qquad  \qquad >> L^2(\Pi_\kappa; \mathcal{H})
\end{CD},
\end{align*}
where $I_H$ is the identity operator on $\mathcal{H}$ and $M_\Phi$ is the operator  defined by 
\[
M_\Phi(\Psi) = \Phi \Psi, \quad \text{for all $\Psi \in L^2(\Pi_\kappa; \mathcal{H})$}
\]
 and $U_o: \ell^2(V_\kappa)_{\rad(o)}\rightarrow L^2(\Pi_\kappa)$ is the unitary operator defined as in \eqref{def-U-o}. 
\end{proposition}

\begin{corollary}
A function  $\Theta: \N_0 \rightarrow B(\mathcal{H})$ induces a bounded radial Toeplitz operator $T_\Theta$ on $\ell^2(V_\kappa; \mathcal{H})$ if and only if there exists $\Phi \in L^\infty(\Pi_\kappa; B(\mathcal{H}))$ such that 
\[
\Theta(n) = \int_\R P_n(t) \Phi(t) d\Pi_\kappa(t), \quad \forall n \ge 0.
\]
Moreover, in this case, the operator norm of $T_\Theta$ is given by
\[
\| T_\Theta\| = \|\Phi\|_{L^\infty(\Pi_\kappa; B(\mathcal{H}))}.
\]
\end{corollary}

\subsection{$\Gamma$-invariant determinantal point processes on $V_\kappa$}

A direct application of our previous results is the classification of a rich family of $\Gamma$-invariant determinantal point processes on $V_\kappa$.

Let us briefly recall the basic materials in the theory of determinantal point processes on discrete sets.  Let $E$ be a countable discrete set, equipped with the counting measure. By a {\it configuration} on $E$, we mean a subset of $E$. Identify the space $\Conf(E)$ of configurations on $E$ with the set $\{0, 1\}^E$ and endow the product topology on the set $\{0, 1\}^E$.   By a {\it point process} on $E$, we mean a Borel probability measure on $\Conf(E)$. For further background on the general theory of point processes, see  Daley and Vere-Jones \cite{DV-1}, Kallenberg \cite{Kallenberg}.

A point process $\PP$ on $E$ is called a {\it determinantal point process} induced by a kernel $K: E\times E \rightarrow \C$ if for any finite subset $\Lambda \subset E$, 
\begin{align}\label{def-DPP}
\PP \Big( \Big\{\xi \in \Conf(E) \big| \xi \supset \Lambda \Big\}\Big) = \det (K(x, y))_{x, y \in \Lambda}.
\end{align}
In the above situation, $K$ is called a {\it correlation kernel} of the point process $\PP$. Clearly,  $\PP$ is uniquely determined by  the kernel $K$ and we denote it by $\PP_K$. Note that different correlation kernels may induce same determinantal point process. 

In the case of Hermitian kernels, a necessary and sufficient condition for a kernel to induce a determinantal point process is provided by Macchi-Soshnikov Theorem, which is also proved independently by Shirai-Takahashi. Here we only state this theorem in the case when $E$ is discrete. 

Recall that an operator $A\in B(\ell^2(E))$ is called positive  and is denoted as $A\ge 0$ if $\langle A v, v\rangle \ge 0$ for all $v \in \ell^2(E)$.
\begin{theorem}[See {\cite{DPP-M, DPP-S,ST-palm2, ST-palm}}]
Let $E$ be a countable discrete set. A Hermitian kernel $K: E\times E\rightarrow \C$ induces a determinantal point process on $E$ if and only if $K$ is the kernel of a {\it positive contractive} bounded operator on  $\ell^2(E)$. 
\end{theorem}
The reader is also referred to Lyons \cite{Lyons-DPP-IHES} for a construction of determinantal point process using exterior algebra and to  Hough-Krishnapur-Peres-Vir{\'a}g \cite{DPP-HKPV} for an alternative method of recursive construction of determinantal point processes with Hermitian correlation kernels. In full generality of non-Hermitian kernels, it is still an open  problem  to give a necessary and sufficient condition on the kernel such that it induces a determinantal point process.

If the set $E$ carries some additional structure (for instance, $E$ is a group), then it is natural to consider determinantal point processes on $E$ arising from positive contractions which are equivariant under the group action of automorphisms of $E$.   For instance,  in  the  discrete case,  Shirai-Takahashi \cite{ST-palm} and   Lyons-Steif \cite{Lyons-stationary} have studied the  determinantal point processes on the group $\Z^d$ or more general countable discrete abelian groups with translation invariant kernels. Lyons-Thom \cite{Lyons-coupling} have studied, among other things, the invariant monotone coupling of two determinantal point processes arising from two equivariant positive contractions $Q_1 \le Q_2$ on the Cayley graph of a sofic group.  In the continuous case, Krishnapur \cite{Krishnapur09} has studied determinantal point processes arising from equivariant orthogonal projections related to holomorphic functions in Poincar\'e disk, the Riemann sphere and the complex plane.

An immediate consequence of Theorem \ref{thm-alg-iso} is the following
\begin{corollary}\label{cor-dpp}
The class of positive contractive radial Toeplitz operators on $\ell^2(V_\kappa)$  is given by 
\[
\Big\{T[\varphi] \Big|  \text{$\varphi \in L^\infty(\Pi_\kappa)$ and $0 \le \varphi \le 1$}\Big\}.
\]
Therefore,  the family of determinantal point processes on $V_\kappa$ induced by $\Gamma$-equivariant  positive contractions  is 
\[
\Big\{\PP_{T[\varphi]}\Big|  \text{$\varphi \in L^\infty(\Pi_\kappa)$ and $0 \le \varphi \le 1$}\Big\}.
\]
\end{corollary}

For $\kappa =1$, the set $V_\kappa$ reduce to the set $\Z$ of integers.  Note that the automorphism group of the graph $\Z$ is generated by the translations and reflection. Positive contractive kernels on $\Z$ which are both translation and reflection invariant are given by the Fourier transform of even functions $\varphi: [-1/2, 1/2]\rightarrow [0,1]$. For instance, the discrete Dyson-sine kernels fall into this class: 
\[
S_\alpha(m, n) := \frac{\sin (\alpha \pi(m-n))}{\pi(m-n)}, \quad m, n \in \Z,
\]
with diagonal understood as $\alpha$ with $\alpha \in [0, 1]$. The kernel $S_\alpha$ induces an orthogonal projection on $\ell^2(\Z)$ and it is known that the determinantal point processes $\PP_{S_\alpha}$ have the Ghosh-Peres number rigidity, see \cite{Ghosh-sine, Ghosh-rigid, BDQ-ETDS}. 

It is natural to ask the following 
\begin{question}\label{ques-rigid}
For $\kappa \ge 2$.  Let\[
J \subset \left[- \frac{2 \sqrt{\kappa} }{ \kappa + 1}, \,  \frac{2 \sqrt{\kappa} }{ \kappa + 1} \right]
\]
be an interval. Does  the determinantal point process  on $V_\kappa$ induced by the  orthogonal projection $T[\mathds{1}_J]$ have the Ghosh-Peres number rigidity? 
\end{question}


\section{$B_\Gamma(\ell^2(V_\kappa))$ is a commutative unital $C^*$-algebra}

\subsection{Proof of Lemma \ref{lem-alg}}
Define  $\delta_0: \N_0 \rightarrow \C$ by 
\begin{align}\label{def-phi-o}
\delta_0(n)  : = 
\left\{
\begin{array}{cc}
1 & n = 0,
\vspace{2mm}
\\ 
0 &  n \ge 1.
\end{array}
\right.
\end{align}
Then the corresponding $\Gamma$-equivariant kernel
\[
T_{\delta_0}(x, y)= \mathds{1}(x = y), \quad  \forall x, y \in V_\kappa
\] 
represents the identity operator $I$ on $\ell^2(V_\kappa)$. As before, we identify any bounded operator on $\ell^2(V_\kappa)$ with its kernel. Thus we have
\begin{align}\label{T-phi-o-id}
T_{\delta_0}  = I \in B_\Gamma(\ell^2(V_\kappa)).
\end{align}
Clearly, $B_\Gamma(\ell^2(V_\kappa)) \subset B(\ell^2(V_\kappa))$ is a linear subspace which is closed under the operation of taking the adjoint. Indeed, we have 
\begin{itemize}
\item  for any $\lambda_1, \lambda_2 \in \C$ and any $T_{\alpha_1}, T_{\alpha_2} \in B_\Gamma(\ell^2(V_\kappa)$,
\begin{align}\label{linear-iota}
\lambda_1 T_{\alpha_1} + \lambda_2 T_{\alpha_2} = T_{\lambda_1 \alpha_1+ \lambda_2 \alpha_2}.
\end{align}
\item for any $T_\alpha \in B_\Gamma(\ell^2(V_\kappa))$,
\begin{align}\label{adj-iota}
T_\alpha^* = T_{\overline{\alpha}}.
\end{align}
\end{itemize} 
It remains to prove that, for any $T_{\alpha_1}, T_{\alpha_2}\in B_\Gamma(\ell^2(V_\kappa))$, there exists $\beta: \N_0\rightarrow \C$ such that 
\begin{align}\label{alpha-alpha-beta}
T_{\alpha_1} T_{\alpha_2} = T_{\alpha_2} T_{\alpha_1} = T_\beta. 
\end{align}

Note first that the assumption $T_{\alpha_1}, T_{\alpha_2} \in B_\Gamma(\ell^2(V_\kappa)) \subset B(\ell^2(V_\kappa))$ implies that for any $x \in V_\kappa$, we have 
\begin{align}\label{bdd-assump}
\sum_{y \in V_\kappa} |\alpha_i(d(x, y))|^2 < \infty, \quad i = 1, 2. 
\end{align}
For proving \eqref{alpha-alpha-beta}, we need to use the following elementary facts.
\begin{itemize}
\item The  group $\Gamma$ acts transitively on $V_\kappa$.
\item Fix any vertex $o \in V_\kappa$. Then for any $g, h \in V_\kappa$ such that $d(o, x) = d(o, y)$, there exists $\gamma \in \Gamma$ such that 
\begin{align}\label{exchange-stab}
\gamma o = o, \quad \gamma x = y.
\end{align}
\item For any $x, y \in V_\kappa$, there exists $\gamma \in \Gamma$ such that
\begin{align}\label{ex-2-pt}
\gamma x = y, \quad \gamma y = x. 
\end{align}
\end{itemize}
Fix a sequence of vertices $(v_n)_{n\ge 0}$ in $V_\kappa$ such that $d(v_k, v_n) = |n-k|$ for any $n, k\ge 0$. For any $x,y \in V_\kappa$, let $ \gamma', \gamma'' \in \Gamma$ be two automorphisms of $V_\kappa$ (whose choices depend on $x, y$) such that 
\[
\gamma'x =v_0, \quad \gamma''v_0= v_0, \quad \gamma''\gamma'y = v_{d(v_0, \gamma' y)}.
\] 
Since $\gamma', \gamma''$ preserve the graph distance,  we have 
\[
d(x, y)  = d(\gamma' x, \gamma'y)= d(v_0, \gamma' y), \quad \gamma'' \gamma' y = v_{d(v_0, \gamma'y)} = v_{d(x, y)}
\]
and
\begin{align}\label{prod-T}
\begin{split}
& T_{\alpha_1} T_{\alpha_2}(x, y) = \sum_{z \in V_\kappa} \alpha_1( d (x, z)) \alpha_2( d(z, y))
\\
 =&   \sum_{z\in V_\kappa} \alpha_1( d (\gamma'x, \gamma'z)) \alpha_2( d(\gamma'z, \gamma'y))  
= \sum_{z' = \gamma'z\in V_\kappa} \alpha_1( d (v_0, z')) \alpha_2( d(z', \gamma'y)) 
\\
  =& \sum_{z' \in V_\kappa} \alpha_1( d (v_0, \gamma'' z')) \alpha_2( d(\gamma'' z', \gamma''\gamma' y))
 =  \sum_{z'' \in V_\kappa} \alpha_1( d (v_0,  z'')) \alpha_2( d( z'', v_{d(x, y)})).
\end{split}
\end{align}
 By \eqref{bdd-assump}, we may define $\beta: \N_0 \rightarrow \C$ by 
\begin{align}\label{def-beta}
\beta(n): = \sum_{z \in V_\kappa} \alpha_1( d (v_0,  z)) \alpha_2( d( z, v_n)).
\end{align}
Then \eqref{prod-T} implies
\[
T_{\alpha_1} T_{\alpha_2}  (x, y)=\beta(d(x, y)) = T_{\beta}(x, y).
\]
Since $x, y \in V_\kappa$ are chosen arbitrarily, we obtain $T_{\alpha_1} T_{\alpha_2}  = T_\beta$.  Similarly, by exchanging $\alpha_1$ and $\alpha_2$, we have $T_{\alpha_2} T_{\alpha_1} = T_{\widetilde{\beta}}$, where $\widetilde{\beta}: \N_0\rightarrow \C$ is defined similarly as in \eqref{def-beta}:
\[
\widetilde{\beta} (n) : = \sum_{z \in V_\kappa} \alpha_2( d (v_0,  z)) \alpha_1( d(z, v_n)).
\]
Now using \eqref{ex-2-pt}, for each $n \ge 0$, we can find $\gamma_n \in \Gamma$ such that $\gamma_n(v_0) = v_n$ and $\gamma_n(v_n) = v_0$. Hence 
\begin{align*}
\widetilde{\beta} (n) &  = \sum_{z \in V_\kappa} \alpha_2( d (\gamma_n(v_0),  \gamma_n( z))) \alpha_1( d(\gamma_n( z), \gamma_n(v_n))) 
\\
& = \sum_{z' =\gamma_n(z)  \in V_\kappa} \alpha_2( d (v_n,   z')) \alpha_1( d( z', v_0)) 
\\
& = \sum_{z  \in V_\kappa} \alpha_2( d (v_n,   z)) \alpha_1( d(z, v_0))  = \beta(n).
\end{align*}
Since $n \ge 0$ is arbitrarily chosen, we obtain $\widetilde{\beta} = \beta$ and complete the proof of the equality \eqref{alpha-alpha-beta}. This completes the proof of Lemma \ref{lem-alg}.

\subsection{A convolution algebra}
Set 
\begin{align}\label{def-A-kappa}
\mathcal{A}(\kappa):  = \Big\{\alpha: \N_0 \rightarrow \C \Big| T_\alpha \in B_\Gamma(\ell^2(V_\kappa)) \Big\}.
\end{align}
We define a convolution-type product $\circledast_\kappa$ on $\mathcal{A}(\kappa)$ as follows.  First fix any sequence $(v_n)_{n\ge 0}$ in $V_\kappa$ satifying 
\begin{align}\label{v-0-v-n}
d(v_k, v_n) = |k-n|, \quad \forall n, k \ge 0. 
\end{align}
By \eqref{bdd-assump}, for any $\alpha_1, \alpha_2 \in \mathcal{A}(\kappa)$,  we may define  $\alpha_1 \circledast_\kappa \alpha_2 \in \mathcal{A}(\kappa)$ by
\begin{align}\label{def-con-prod}
(\alpha_1 \circledast_\kappa \alpha_2) (n): =  \sum_{x \in V_\kappa} \alpha_1(d(v_0, x)) \alpha_2(d(x, v_n)), \quad \forall n \ge 0. 
\end{align}
Indeed, from the proof of Lemma \ref{lem-alg}, we have
\begin{align}\label{T-alpha-prod}
T_{\alpha_1 \circledast_\kappa \alpha_2}= T_{\alpha_1} T_{\alpha_2} = T_{\alpha_2 \circledast_\kappa \alpha_1} = T_{\alpha_2} T_{\alpha_1}.
\end{align}
Therefore, we do have $\alpha_1 \circledast_\kappa \alpha_2 \in \mathcal{A}(\kappa)$.  In particular,  \eqref{T-alpha-prod} implies also that the definition \eqref{def-con-prod} does not depend on the choices of the sequence $(v_n)_{n \ge 0}$ with the property \eqref{v-0-v-n} and  we have a natural isomorphism between the algebras $\mathcal{A}(\kappa)$ and $B_\Gamma(\ell^2(V_\kappa))$ given by 
\begin{align}\label{iso-A-B}
\begin{array}{cccc}
j: & \mathcal{A}(\kappa) &  \xrightarrow{\, \simeq \, } &  B_\Gamma(\ell^2(V_\kappa))
\\
 & \alpha & \mapsto & T_\alpha
\end{array}.
\end{align}
The following explicit computation of $\alpha_1 \circledast_\kappa \alpha_2$ will be useful later. 
\begin{lemma}\label{lem-convolution}
For any $\alpha_1, \alpha_2 \in \mathcal{A}(\kappa)$, we have 
\begin{align}\label{comp-alpha-alpha}
(\alpha_1  \circledast_\kappa \alpha_2)(n) = 
\left\{
\begin{array}{lc}
{\displaystyle \alpha_1(0) \alpha_2(0) + \sum_{l = 1}^\infty (\kappa+1) \kappa^{l-1}  \cdot  \alpha_1 (l) \alpha_2(l),} & \text{if $n = 0$}
\vspace{2mm}
\\
{\displaystyle \sum_{l = 0}^\infty  \kappa^l[\alpha_1(l)\alpha_2(l+1) + \alpha_1(l+1) \alpha_2(l)],} & \text{if $n = 1$}
\vspace{3mm}
\\
 {\displaystyle 
\sum_{l = 0}^\infty  \kappa^l[\alpha_1(l)\alpha_2(l+n) + \alpha_1(l+n) \alpha_2(l)] + \sum_{i =1}^{n-1}    \alpha_1 (i) \alpha_2(n-i)} & \text{if $n \ge 2$}
\\
 \qquad  \quad  +{\displaystyle \sum_{i =1}^{n-1}   \sum_{l =1}^\infty  (\kappa-1) \kappa^{l-1}  \alpha_1(l+i) \alpha_2( l +  (n-i)),} & 
\end{array}
\right..
\end{align}
\end{lemma}

\begin{proof}
Let $\alpha_1, \alpha_2 \in \mathcal{A}(\kappa)$. Fix a sequence $(v_n)_{n\ge 0}$ in $V_\kappa$ satisfying \eqref{v-0-v-n}. If $n = 0$, we have 
\[
(\alpha_1 \circledast_\kappa \alpha_2)(0)  =\sum_{x \in V_\kappa} \alpha_1(d(v_0, x)) \alpha_2 ( d(v_0, x)) = \sum_{l = 0}^\infty \#\{x\in V_\kappa| d(v_0, x) = l\} \cdot \alpha_1(l) \alpha_2(l). 
\]
Since 
\[
\#\{x\in V_\kappa| d(v_0, x) = l\}  = 
\left\{
\begin{array}{lc}
 1,& \text{if $l = 0$}
\\
(\kappa+1) \kappa^{l-1}, & \text{if $l \ge 1$}
\end{array}
\right.,
\]
we obtain the desired equality \eqref{comp-alpha-alpha} for $n =0$.

If $n =1$, then we have a partition of $V_\kappa$:  
\[
V_\kappa = W_0 \sqcup W_1,
\] obtained as follows:  erase the edge $\wideparen{v_0v_1}$ from the tree  $\ct_\kappa$ to get a forest consisting of two components, for $i =0$ or $i = 1$, let $W_i \subset V_\kappa$ be the subset consists exactly of the vertices of the  connected component containing $v_i$. It is easy to see 
\[
\left\{\begin{array}{cc}
d(x, v_1) = d(x, v_0) + 1, & \forall x \in W_0,
\vspace{2mm}
\\
d(x, v_0)  = d(y, v_1) +1, & \forall x \in W_1.
\end{array}
\right.
\] 
Moreover, for any integer $l \ge 0$, 
\[
\#\{ x \in W_i | d(v_i, x) = l\} =  \kappa^l, \quad i = 0, 1.
\]
Therefore, 
\begin{align*}
(\alpha_1 \circledast_\kappa \alpha_2)(1)  = &  \sum_{x \in W_0}  \alpha_1(d(v_0, x)) \alpha_2(d (x, v_1))  + \sum_{x \in W_1}  \alpha_1(d(v_0, x)) \alpha_2(d (x, v_1)) 
\\
= & \sum_{l=0}^\infty  \kappa^l  \cdot \alpha_1(l) \alpha_2(l+1) +   \sum_{l = 0}^\infty \kappa^l \cdot \alpha_1(l+1)   \alpha_2(l). 
\end{align*}
This completes the proof of the desired equality \eqref{comp-alpha-alpha} for $n =1$.

Finally, for $n \ge 2$, we have the following partition 
\[
V_\kappa = W_0' \sqcup W_1' \sqcup \cdots \sqcup W_n', 
\]
obtained as follows: erase all the edges $\wideparen{v_0v_1}, \wideparen{v_1v_2}, \cdots, \wideparen{v_{n-1}v_n}$ to get a forest consisting of $n+1$ connected compents,  then for all $0 \le i \le n$, the subset $W_i \subset V_\kappa$ consists exactly of the vertices of the connected component containing $v_i$. We have the following properties on $W_i$'s: 
\[
\left\{ 
\begin{array}{ll}
d(x, v_n)  = d(x, v_0) + n, & \text{if $x \in W_0'$,} 
\vspace{2mm}
\\
d(x, v_0) = d(x, v_n) + n, &\text{if $x\in W_n'$,}
\vspace{2mm}
\\
d(x, v_0) -i = d(x, v_n)- (n-i) \ge 0,  & \text{if $1\le i \le n - 1$ and $g \in W_i'$.} 
\end{array}
\right.
\]
Moreover, we have 
\[
\left\{
\begin{array}{lc}
\#\{x \in W_i' | d(v_i, x) = l\} =  \kappa^l &  \text{$\forall l \ge 0$ and $i \in \{0,  n\}$},
\vspace{2mm}
\\
\#\{x \in W_i' | d(v_0, x) = l + i \} = 
\left\{ 
\begin{array}{lc}
 1, & l  =0
\\
(\kappa-1) \kappa^{l-1} , & l \ge 1 
\end{array}
\right.
& \text{for all $1 \le i \le n-1$.}
\end{array}
\right.
\]
Therefore, 
\begin{align*}
(\alpha_1 \circledast_\kappa \alpha_2)(n)  =&  \sum_{0 \le i \le n }\sum_{x \in W_i} \alpha_1(d(v_0, x)) \alpha_2( d(x, v_n)) 
\\
=&  \sum_{l = 0}^\infty \kappa^l \alpha_1(l) \alpha_2( l + n) + \sum_{l=0}^\infty \kappa^l \alpha_1(l+n) \alpha_2(l) 
\\
& +  \sum_{i =1}^{n-1}   \Big( \alpha_1 (i) \alpha_2(n-i) +  \sum_{l =1}^\infty  (\kappa-1) \kappa^{l-1}  \alpha_1(l+i) \alpha_2( l +  (n-i)) \Big).
\end{align*}
This completes the proof of the desired equality \eqref{comp-alpha-alpha} for $n \ge 2$.
 \end{proof}

\section{Description of the set $B_\Gamma(\ell^2(V_\kappa))$}

\begin{proposition}\label{thm-bdd}
We have the following set-theoretical equality 
\[
B_\Gamma(\ell^2(V_\kappa))  = \Big\{T[\varphi] \Big| \varphi \in L^\infty(\Pi_\kappa)\Big\}.
\]
\end{proposition}

As before, let $(P_n)_{n\ge 0}$ be the sequence of Cartier-Dunau polynomials defined in \eqref{cartier-dunau}. 
{\flushleft \bf Elementary observation (E-O)}: 
Since the linear span of $(P_n)_{n\ge 0}$ are dense in the space $C([-1, 1])$ of continuous function on the interval $[-1,1]$, any Radon measure $\mu$ on $[-1, 1]$ is uniquely determined by  the sequence of complex numbers: 
\[
\Big(\int_{[-1, 1]} P_n(t) d\mu(t)\Big)_{n \ge 0}.
\] 

  Recall that a kernel $K: V_\kappa \times V_\kappa \rightarrow \C$ is called positive definite if for any finite subset $\Lambda \subset V_\kappa$ and any $(v_x)_{x \in \Lambda} \in \C^\Lambda$, we have 
\[
\sum_{x, y \in \Lambda} K (x, y) v_x \overline{v_y} \ge 0. 
\]

Recall the definition \eqref{def-T-phi} of radial Toeplitz kernels on $V_\kappa$. 
Our proof is based on the following Arnaud's classification \cite{Arnaud94} of positive definite radial Toeplitz kernels on $V_\kappa$.

\begin{theorem}[{\cite[Theorem 1.2]{Arnaud94}}]\label{thm-arn}
A radial Toeplitz kernel $T_\alpha: V_\kappa \times V_\kappa \rightarrow \C$ is positive definite if and only if there exists a finite positive measure $\mu_\alpha$ on the interval $[-1, 1]$ with
\begin{align}\label{phi-mu-phi}
\alpha(n) = \int_{[-1, 1]} P_n(t) d\mu_\alpha(t), \quad \forall n \ge 0.
\end{align}
\end{theorem}

Recall the definition of $\delta_0$  in \eqref{def-phi-o}. Since the radial Toeplitz kernel 
\[
T_{\delta_0}(x, y)= \mathds{1}(x = y), \quad \forall x, y \in V_\kappa
\]
is positive definite, by Theorem \ref{thm-arn}, there exists a unique  finite  positive measure $\mu_{\delta_0}$ on the interval $[-1, 1]$ corresponding to $\delta_0$ via the equality \eqref{phi-mu-phi}. 

\begin{lemma}\label{lem-id}
We have $\mu_{\delta_0} = \Pi_\kappa$.  As a consequence,  
\begin{align}\label{T-1-id}
T[1] = I,
\end{align}
where $I$ stands for the identity operator on $\ell^2(V_\kappa)$.
\end{lemma}
\begin{proof}
Recall that $(P_n)_{n\ge 0}$ is a sequence of orthogonal polynomials for the probability measure $\Pi_\kappa$ and the support of the probability measure $\Pi_\kappa$ is given by
\[
\supp (\Pi_\kappa) = [-\frac{2 \sqrt{\kappa}}{ \kappa + 1}, \frac{2 \sqrt{\kappa}}{ \kappa + 1}] \subset [-1, 1]. 
\]
Hence, by recalling that $P_0(x) = 1$, we have 
\[
\int_{[-1, 1]} P_n(t) d\Pi_\kappa (t) = \int_\R P_n(t) d\Pi_\kappa(t)  = \int_\R P_0(t) P_n(t) d\Pi_\kappa(t)  = \delta_0(n). 
\]
The above equality implies on the one hand the equality $\mu_{\delta_0} = \Pi_\kappa$ and on the other hand the equality $T[1] = T_{\delta_0}$. By \eqref{T-phi-o-id}, we have $T[1] = I$.  
\end{proof}

We shall also use the following simple and elementary lemma. 
\begin{lemma}\label{lem-norm-pos}
Let $A$ be a Hermitian bounded operator on a complex Hilbert space $\mathcal{H}$. Then the operator norm $\| A\|$ is given by 
\begin{align}\label{norm-A}
\| A\| = \min\Big\{ \lambda  \ge 0\Big| \text{$\lambda I +A \ge 0$ and $\lambda I - A \ge 0$}  \Big\},
\end{align}
where $I$ is the identity operator on $\mathcal{H}$. 

Moreover, if $\mathcal{H} = \ell^2(E)$ with $E$ a discrete countable set, then a Hermitian kernel $K: E\times E\rightarrow \C$ induces a bounded operator on $\ell^2(E)$ if and only if there exists a finite number $\lambda \ge 0$ such that  both kernels 
\[
E \times E \ni (x, y) \mapsto  (\lambda I \pm K)(x, y): = \lambda \mathds{1}(x = y) \pm K(x,y) \in \C
\]
are positive definite. In this case, the operator norm of the induced bounded operator is given by the smallest $\lambda \ge 0$ making both kernel $\lambda I \pm K$ positive definite. 
\end{lemma}
\begin{proof}
To prove the equality \eqref{norm-A} for  a Hermitian bounded operator $A$, it suffices to show $\| A\| \cdot I  \pm A \ge 0$ for any Hermitian bounded operator $A$ on $\mathcal{H}$.  But this is an immediate consequence of the following inequality:  
\[
\langle  ( \| A\| \cdot I  \pm A) v, v \rangle =  \| A\| \| v\|^2 \pm \langle A v, v\rangle \ge 0, \quad \forall v\in \mathcal{H}.
\]

The remaining assertions for  $\mathcal{H}= \ell^2(E)$ follow from  \eqref{norm-A} and the following classical argument: for any exhausting sequence $(E_n)_{n\ge 0}$ of finite subsets of $E$, that is, $E_n \subset E_{n + 1}$ and $\bigcup_{n\ge 1} E_n = E$, if we denote by $K_{E_n}$ the finite matrix obtained by truncation onto $E_n\times E_n$ of the kernel $K$, then 
\[
\| K\| = \sup_{n \ge 1} \| K_{E_n}\|,
\] 
where we use the convention that $\| K\| = \infty$ if $K$ does not induces a bounded operator on $\ell^2(E)$. 
\end{proof}

\begin{proof}[Proof of Proposition \ref{thm-bdd}]
Assume that $T_{\alpha} \in B_\Gamma(\ell^2(V_\kappa))$. Using the elementary identities
\[
T_{\Re(\alpha)} = \frac{T_\alpha + T_\alpha^*}{2}\in B_\Gamma(\ell^2(V_\kappa)), \quad T_{\Im(\alpha)}  = \frac{T_\alpha  - T_{\alpha}^*}{2}\in B_\Gamma(\ell^2(V_\kappa)),
\]
we may assume without loss of generality that $T_\alpha$ is self-adjoint or equivalently, $\alpha$ is real-valued. If  $T_\alpha$ is a Hermitian bounded operator and if we denote $\lambda  = \| T_\alpha\|\ge 0$, then by Lemma \ref{lem-norm-pos},  we have  $\lambda I \pm T_\alpha \ge 0$.  Thus by \eqref{T-phi-o-id},  we obtain the operator order inequalities
\[
T_{\lambda \delta_0 \pm \alpha} \ge 0. 
\]
This implies that the corresponding kernels $T_{\lambda \delta_0 \pm \alpha}$ are positive definite and by Theorem \ref{thm-arn},  there exist positive Radon measures $\mu_{\lambda \delta_0 \pm \alpha}$ on $[-1, 1]$, which are uniquely determined by $\lambda \delta_0 + \alpha$ and $\lambda \delta_0 - \alpha$ respectively,   such that 
\[
\lambda \delta_0(n) \pm \alpha(n) = \int_{[-1, 1]} P_n  d\mu_{\lambda \delta_0 \pm \alpha}, \quad \forall n \ge 0. 
\]
Set 
\[
\mu_\alpha: = \frac{\mu_{\lambda \delta_0 +\alpha} - \mu_{\lambda \delta_0 - \alpha}}{2}.
\]
Then $\mu_\alpha$ is a signed Radon measure on $[-1, 1]$ such that  
\[
\alpha(n) = \int_{[-1, 1]} P_n  d\mu_{\alpha}, \quad \forall n \ge 0.
\]
Furthermore, by Lemma \ref{lem-id},   we obtain
\[
\pm \alpha(n) = \int_{[-1, 1]} P_n   d (\mu_{\lambda \delta_0 \pm \alpha} -  \lambda \Pi_\kappa), \quad \forall n \ge 0. 
\]
Thus by the elementary observation (E-O), we have 
\begin{align}\label{mu-alpha-2-way}
\mu_\alpha = \mu_{\lambda \delta_0 + \alpha} -  \lambda \Pi_\kappa =  \lambda \Pi_\kappa - \mu_{\lambda \delta_0 - \alpha}. 
\end{align}
Hence 
\begin{align}\label{sum-2-meas}
\mu_{\lambda \delta_0 + \alpha} + \mu_{\lambda \delta_0 - \alpha}  = 2  \lambda \Pi_\kappa. 
\end{align}
But recall that both Radon measures $\mu_{\lambda \delta_0 \pm \alpha}$ are positive measures. Therefore, the equality \eqref{sum-2-meas} implies that both measures   $\mu_{\lambda \delta_0 \pm \alpha}$ are absolutely continuous with respect to the probability measure $\Pi_\kappa$. If we denote the corresponding Radon-Nikodym derivatives by
\[
\rho_{\lambda \delta_0 \pm \alpha} (t): = \frac{d \mu_{\lambda \delta_0 \pm \alpha}}{d \Pi_\kappa}(t) \ge 0, \quad   \text{for $\Pi_{\kappa}$-a.e. $t \in  \supp(\Pi_\kappa) \subset [-1, 1]$},
\]
Then theg equality \eqref{sum-2-meas} implies also that 
\[
\rho_{\lambda \delta_0 + \alpha} (t) + \rho_{\lambda \delta_0 - \alpha} (t) = 2 \lambda,  \quad  \text{for $\Pi_{\kappa}$-a.e. $t$}.
\]
Hence $\rho_{\lambda \delta_0 \pm \alpha} \in L^\infty(\Pi_\kappa)$. Finally, by \eqref{mu-alpha-2-way}, we have 
\[
\mu_\alpha = \rho_{\lambda \delta_0  + \alpha} \Pi_\kappa - \lambda \Pi_\kappa  = (\rho_{\lambda \delta_0 + \alpha} - \lambda) \Pi_\kappa. 
\]
In other words, by setting $\varphi  = \rho_{\lambda \delta_0 + \alpha} - \lambda \in L^\infty(\Pi_\kappa)$, we obtain 
\[
\alpha(n) = \int_{[-1, 1]} P_n \varphi d\Pi_\kappa =  \int_\R P_n \varphi d\Pi_\kappa, \quad \forall n \ge 0. 
\]
That is, we have $T_\alpha = T[\varphi]$. Thus 
\[
B_\Gamma(\ell^2(V_\kappa))  \subset \Big\{T[\varphi] \Big| \varphi \in L^\infty(\Pi_\kappa)\Big\}.
\]

Conversely, if $\varphi \in L^\infty(\Pi_\kappa)$, we want to show that $T[\varphi] \in B_\Gamma(\ell^2(V_\kappa)$. By writing $\varphi = \Re(\varphi) + i \Im (\varphi)$ and using $T[\varphi] = T[\Re(\varphi)] + i T[\Im (\varphi)]$, we may assume without loss of generality that $\varphi$ is real-valued.  Let $\|\varphi\|_\infty = \lambda \ge 0$, then since $\varphi$ is real-valued, we have $\lambda \pm \varphi(t) \ge 0$ for $\Pi_\kappa$-a.e. $t$. Therefore, by Theorem \ref{thm-arn} and \eqref{T-1-id}, the two kernels
\[
T[\lambda \pm \varphi] = \lambda I \pm T[\varphi]
\] 
are both positive definite. Hence by Lemma \ref{lem-norm-pos}, $T[\varphi]$ is a bounded operator on $\ell^2(V_\kappa)$.  This implies the desired inclusion 
\[
 \Big\{T[\varphi] \Big| \varphi \in L^\infty(\Pi_\kappa)\Big\} \subset B_\Gamma(\ell^2(V_\kappa)).
\]
\end{proof}

\section{$B_\Gamma(\ell^2(V_\kappa))$ is isomorphic to $L^\infty(\Pi_\kappa)$}

In this section, we are going to prove Theorem \ref{thm-alg-iso}.
\subsection{Outline of the strategy}
Before proceeding to the proof of Theorem \ref{thm-alg-iso}., we briefly explain our strategy. Recall the map $\iota: L^\infty(\Pi_\kappa) \rightarrow B_\Gamma(\ell^2(V_\kappa))$ defined in \eqref{def-iota}: $\iota(\varphi) = T[\varphi]$.
\begin{itemize}
\item Clearly, by Proposition \ref{thm-bdd}, the Elementary Observation (E-O) and the equalities \eqref{linear-iota} and \eqref{adj-iota}, the map $\iota$ is a linear bijection between $L^\infty(\Pi_\kappa) $ and  $B_\Gamma(\ell^2(V_\kappa))$ which preserves the operation of the corresponding involutions on both spaces. 
\item By  the equality $T[1] = I$ obtained in \eqref{T-1-id}, the map $\iota$ is unital. 
\item If $\varphi \in L^\infty(\Pi_\kappa)$ is an element with $\varphi \ge 0$, then by Theorem \ref{thm-arn} and Proposition \ref{thm-bdd}, the operator $T[\varphi]$ is positive. Hence $\iota$ is positive.
\item A classical result in the theory of $C^*$-algebras says that any positive unital linear map between unital $C^*$-algebras is automatically contractive and thus norm-continuous. This fact however will not be directly used in this paper. 
\item
Therefore, it remains to prove the  key multiplicative property
\begin{align}\label{mul-property}
T[\varphi_1 \varphi_2] = T[\varphi_1]T[\varphi_2], \quad \forall \varphi_1, \varphi_2 \in L^\infty(\Pi_\kappa). 
\end{align}
\item We shall first prove in Lemma \ref{lem-mul-poly} below that the multiplicative property \eqref{mul-property} for all polynomial functions $\varphi_1, \varphi_2 \in \C[t]$.
\item Although $\iota$ is norm-continuous, the space $\C[t]$ of polynomials is however not norm dense in $L^\infty(\Pi_\kappa)$. But $\C[t] \subset L^\infty(\Pi_\kappa)$ is dense in the weak-star-topology $\sigma(L^\infty, L^1)$ on $L^\infty(\Pi_\kappa)$.   Therefore, for extending the multiplicative property \eqref{mul-property} for polynomials to all functions in $L^\infty(\Pi_\kappa)$, we shall use certain continuity result with respect to the weak-star topology on $L^\infty(\Pi_\kappa)$.   However, using this approach, we will have to work with certain topology in $B_\Gamma(\ell^2(V_\kappa))$  and need to deal with the continuity of the operation of  the operator-product in  $B_\Gamma(\ell^2(V_\kappa))$. That is, we need to study the continuity of the map 
\begin{align}\label{op-prod-map}
B_\Gamma(\ell^2(V_\kappa)) \times B_\Gamma(\ell^2(V_\kappa))  \ni (T[\varphi_1], T[\varphi_2]) \mapsto T[\varphi_1]T[\varphi_2] \in B_\Gamma(\ell^2(V_\kappa))
\end{align}
with respect to the chosen topology.
\item For avoiding working with the continuity of the map \eqref{op-prod-map},  we work with certain topology on the algebra $\mathcal{A}(\kappa)$ defined in \eqref{def-A-kappa} and work with the convolution-type product $\circledast_\kappa$ defined in \eqref{comp-alpha-alpha}.  See Lemma \ref{lem-cp-cont} below. 
\end{itemize}

\subsection{Details of the proof}

Now we proceed to the proof of Theorem \ref{thm-alg-iso}. Recall the definition \eqref{def-hat} for all $\varphi\in L^1(\Pi_\kappa)$:
\[
\widehat{\varphi}(n): = \int_\R P_n \varphi d\Pi_\kappa. 
\]
Recall also the definition \eqref{def-A-kappa}  of $\mathcal{A}(\kappa)$ and the definition \eqref{comp-alpha-alpha} for the convolution-type product $\circledast_\kappa$ on $\mathcal{A}(\kappa)$. 
  
\begin{lemma}\label{lem-mul-poly}
For any $\varphi_1, \varphi_2 \in \C[t]$, we have 
\[
T[\varphi_1 \varphi_2] = T[\varphi_1]T[\varphi_2]. 
\]
Therefore, for any $\varphi_1, \varphi_2 \in \C[t]$, we have 
\[
\widehat{\varphi_1\varphi_2} = \widehat{\varphi_1} \circledast_\kappa \widehat{\varphi_2}.
\]
\end{lemma}

\begin{lemma}\label{lem-cp-cont}
For any integer $n \ge 0$ and any $\varphi_0 \in L^\infty(\Pi_\kappa)$, the map 
\begin{align}\label{varphi-to-dast-n}
L^\infty(\Pi_\kappa) \ni \varphi \mapsto (\widehat{\varphi} \circledast_\kappa \widehat{\varphi_0})  (n)\in \C
\end{align}
is continuous with respect to the weak-star topology on $L^\infty(\Pi_\kappa)$. 
\end{lemma}

We shall compare Lemma \ref{lem-cp-cont} with the following parallel but obvious lemma (in fact, after proving Theorem \ref{thm-alg-iso}, the maps  \eqref{varphi-to-dast-n} and \eqref{varphi-to-n} are the same). 
\begin{lemma}\label{lem-trivial}
For any integer $n \ge 0$ and any $\varphi_0 \in L^\infty(\Pi_\kappa)$, the map 
\begin{align}\label{varphi-to-n}
L^\infty(\Pi_\kappa) \ni \varphi \mapsto \widehat{\varphi \varphi_0}(n) = \int_\R P_n \varphi \varphi_0 d\Pi_\kappa \in \C
\end{align}
is continuous with respect to weak-star-topology on $L^\infty(\Pi_\kappa)$.
\end{lemma}

\begin{proof}[Proof of  Theorem \ref{thm-alg-iso} by using Lemma \ref{lem-mul-poly} and Lemma \ref{lem-trivial}]
As we explained in the beginning of this section, it suffices to prove \eqref{mul-property}. 

Now fix any pair  $\varphi_1, \varphi_2 \in L^\infty(\Pi_\kappa)$. Recall the algebra-isomorphism \eqref{iso-A-B} between $\mathcal{A}(\kappa)$ and $B_\Gamma(\ell^2(V_\kappa))$. We have
\[
T[\varphi_1 \varphi_2] = T_{\widehat{\varphi_1 \varphi_2}}, \quad T[\varphi_1] T_{\varphi_2} = T_{\widehat{\varphi_1}} T_{\widehat{\varphi_2}} = T_{\widehat{\varphi_1} \circledast_\kappa \widehat{\varphi_2}}.
\]
It suffices to show that 
\begin{align}\label{lim-to-prove}
\widehat{\varphi_1 \varphi_2} = \widehat{\varphi_1} \circledast_\kappa \widehat{\varphi_2}. 
\end{align}
For $i = 1,2$, let $(\varphi_{i, l})_{l\ge 0}$ be a sequence of polynomials such that $\varphi_{i, l}$ converges to $\varphi_i$ with respect to the weak-star topology of $L^\infty(\Pi_\kappa)$ and $\| \varphi_{i, l}\|_\infty \le \| \varphi_i\|_\infty$ (here we can take sequences instead of nets because the  weak-star topology on the unit ball of $L^\infty(\Pi_\kappa)$ is metrizable). Then by Lemma \ref{lem-mul-poly}, we have 
\begin{align}\label{id-before-lim}
\widehat{\varphi_{1, l} \varphi_{2, m}} = \widehat{\varphi_{1, l}} \circledast_\kappa \widehat{\varphi_{2, m}}, \quad \forall l, m \ge 0.
\end{align}
We then fix any integer $n \ge 0$. Apply Lemma \ref{lem-trivial} twice, we obtain 
\begin{align}\label{2-lim-p}
\lim_{l\to\infty}  \Big( \lim_{m\to\infty}   \widehat{\varphi_{1, l} \varphi_{2, m}}  (n) \Big) = \lim_{l\to\infty}    \widehat{\varphi_{1, l} \varphi_{2}}  (n) =  \widehat{\varphi_{1} \varphi_{2, m}}  (n).
\end{align}
Apply Lemma \ref{lem-cp-cont} twice (note that $\circledast_\kappa$ makes $\mathcal{A}(\kappa)$ a commutative algebra), we obtain 
\begin{align}\label{2-lim-cp}
\lim_{l\to\infty}  \Big( \lim_{m\to\infty}  \widehat{\varphi_{1, l}} \circledast_\kappa \widehat{\varphi_{2, m}}   (n) \Big)  = \lim_{l\to\infty}   \widehat{\varphi_{1, l}} \circledast_\kappa \widehat{\varphi_{2}}   (n) =   \widehat{\varphi_{1}} \circledast_\kappa \widehat{\varphi_{2}}   (n).
\end{align}
Combining \eqref{id-before-lim}, \eqref{2-lim-p} and \eqref{2-lim-cp}, we obtain 
\[
\widehat{\varphi_1 \varphi_2} (n) = \widehat{\varphi_1} \circledast_\kappa \widehat{\varphi_2} (n). 
\]
Since $n \ge 0$ is chosen arbitrarily, we complete the proof of the equality \eqref{lim-to-prove} and thus complete the whole proof of Theorem \ref{thm-alg-iso}. 
\end{proof}

It remains to prove Lemma \ref{lem-mul-poly} and Lemma \ref{lem-cp-cont} (note that as we said before, Lemma \ref{lem-trivial} is trivial and follows directly from the definition of weak-star topology). 

\medskip
{ \flushleft \bf On Lemma \ref{lem-mul-poly}.}
Since  $T[\varphi_1\varphi_2]$ and $T[\varphi_1]T[\varphi_2]$ depends on $\varphi_1, \varphi_2$ bilinearly, to prove Lemma \ref{lem-mul-poly}, it suffices to prove that for any integers $l, m \ge 0$, 
\begin{align}\label{T-x-lm}
T[t^{l + m}] = T[t^l] T[t^m].
\end{align}
Recall  the equality  \eqref{T-1-id} that $T[1] = I$. Therefore,  the equality \eqref{T-x-lm} holds for $l = 0$ and any $m \ge 0$. 

Clearly, the equalities for all $l \ge 1, m \ge 0$ follow from the following particular case 
\begin{align}\label{T-x-one-m}
T[t^{1 + m}] = T[t] T[t^m], \quad \forall m \ge 0.
\end{align}
Indeed, we can use the following induction argument.
\begin{itemize}
\item Assume first that the equalities \eqref{T-x-one-m} hold   for all $m \ge 0$. In other words, the equalities \eqref{T-x-lm} hold for $l = 1$ and for all $m \ge 0$. 
\item  If for a fixed integer $l \ge 1$, the equalities \eqref{T-x-lm}  hold for all $m \ge 0$. Then we can show that these equalities hold for all $m \ge 0$ by replacing $l$ by $l +1$. This is done as follows.  Recall that $B_\Gamma(\ell^2(V_\kappa))$ is commutative. For any $m \ge 0$ , we have 
\[
T[t^{l+1 + m}]  \overset{\text{induction hypothesis}}{=\joinrel=\joinrel=\joinrel= \joinrel=\joinrel= \joinrel=\joinrel=\joinrel=\joinrel=\joinrel=}  T[t^l] T[t^{1 +m}] \overset{\eqref{T-x-one-m}}{=\joinrel=\joinrel=} T[t^l]  T[t] T[t^m]  \overset{\eqref{T-x-one-m}}{=\joinrel=\joinrel=} T[t^{l+1}]T[t^m].
\]
\end{itemize}

It remains to prove \eqref{T-x-one-m}. Using the isomorphism \eqref{iso-A-B} between $\mathcal{A}(\kappa)$ and $B_\Gamma(\ell^2(V_\kappa))$, the equalities \eqref{T-x-one-m} are equivalent to 
\begin{align}\label{alpha-one-m}
\beta_{1 + m} = \beta_1 \circledast_\kappa \beta_m, \quad \forall m \ge 0,
\end{align}
where $\beta_m\in \mathcal{A}(\kappa)$ is a function $\beta_m: \N_0 \rightarrow \C$ defined by 
\[
\beta_m (n) = \widehat{t^m}(n): = \int_{\R} P_n(t) t^m d\Pi_\kappa(t), \quad \forall n \ge 0. 
\]
Let us  compute explictly $\beta_1\in \mathcal{A}(\kappa)$.  Recall first that 
\[
\beta_0 (n) = \widehat{1}(n) = \delta_0(n). 
\]
Using the definition, if $n =0$, then recalling $P_0(t)\equiv 1$ and $P_1(t) = t$, we have 
\[
\beta_1(0)  = \int_\R P_0(t) t  d\Pi_\kappa(t) = \int_\R P_0(t) P_1(t)  d\Pi_\kappa(t)  = \langle P_0, P_1\rangle_{L^2(\Pi_\kappa)} = 0.
\]
For $n \ge 1$, recalling the recursion relation \eqref{cartier-dunau}, we have 
\begin{align*}
\beta_1(n) & = \int_\R P_n(t) t d\Pi_\kappa(t) = \int_\R \Big[ \frac{\kappa}{ \kappa + 1} P_{n+1} (t) + \frac{1}{\kappa+1} P_{n-1} (t)\Big] d\Pi_\kappa(t) 
\\
&  = \frac{\kappa}{\kappa + 1} \delta_0(n+1) +  \frac{1}{\kappa + 1} \delta_0(n-1) = \frac{1}{\kappa +1} \delta_1(n), \quad \forall n \ge 1.
\end{align*}
Combining the above two equalities, we obtain 
\begin{align}\label{expl-alpha-one}
\beta_1(n) = \frac{1}{\kappa + 1} \delta_1(n), \quad \forall n \ge 0. 
\end{align}
Then using the recursion relation \eqref{cartier-dunau}, for $n \ge 1$,  we have 
\begin{align}\label{alpha-one-m-n}
\begin{split}
\beta_{1 + m}(n) & = \int_\R P_n(t) t^{1 + m} d\Pi_\kappa(t) = \int_\R t^m \Big[ \frac{\kappa}{ \kappa + 1} P_{n+1} (t) + \frac{1}{\kappa+1} P_{n-1} (t)\Big]  d\Pi_\kappa(t) 
\\
& = \frac{\kappa}{\kappa+1} \beta_m(n+1) + \frac{1}{\kappa + 1} \beta_m(n-1).
\end{split}
\end{align}
If $n=0$, we have 
\begin{align}\label{alpha-one-m-zero}
\beta_{1 + m}(0) = \int_\R t^{1 + m} d\Pi_\kappa(t) = \int_\R P_1(t) t^m d\Pi_\kappa(t) =  \beta_m (1).
\end{align}
 
On the other hand, using the explicit formula \eqref{comp-alpha-alpha} for computing the $\circledast_\kappa$ product and the equality $\beta_1 = \frac{1}{\kappa + 1} \delta_1$, we obtain
\begin{align}\label{expl-conv-p}
(\beta_1 \circledast_\kappa \beta_m) (n) = \left\{
\begin{array}{lc}
{\displaystyle  \beta_m(1),} & \text{if $n = 0$}
\vspace{2mm}
\\
{\displaystyle    \frac{\kappa}{\kappa+1} \beta_m(2) + \frac{1}{\kappa + 1} \beta_m (0),} & \text{if $n = 1$}
\vspace{3mm}
\\
 {\displaystyle \frac{\kappa}{\kappa + 1} \beta_m(1+n) ] +  \frac{1}{\kappa+1} \beta_m(n-1)} & \text{if $n \ge 2$}
\end{array}
\right..
\end{align}
By comparing \eqref{alpha-one-m-n}, \eqref{alpha-one-m-zero} with \eqref{expl-conv-p}, we obtain the equality 
\[
\beta_{1 + m} (n) =  (\beta_1 \circledast_\kappa \beta_m )(n), \quad \forall n\ge 0.
\]
This completes the proof of the equality \eqref{expl-alpha-one} and thus completes the whole proof of Lemma \ref{lem-mul-poly}.

\medskip
{\flushleft \bf On Lemma \ref{lem-cp-cont}.}
For proving Lemma \ref{lem-cp-cont}, we need to compute the $L^2(\Pi_\kappa)$-norms $\| P_n\|_{L^2(\Pi_\kappa)}$ of the Cartier-Dunau polynomials. 

\begin{lemma}\label{lem-L2-norm}
The $L^2(\Pi_\kappa)$-norms of the Cartier-Dunau polynomials are given by 
\[
\| P_n\|_{L^2(\Pi_\kappa)} =  
\left\{
\begin{array}{cc}
 1& \text{if $n =0$}
\vspace{2mm}
\\ 
\frac{1}{\sqrt{\kappa^{n-1} (\kappa +1)}} & \text{if $n \ge 1$}
\end{array}
\right..
\]
\end{lemma}

\begin{proof}
Recall that the leading coefficient $k_n$ of the Cartier-Dunau polynomial $P_n$ is given in \eqref{leading-coef}: 
\[
k_0 = 1 \text{\, and \,} k_n =     \left(\frac{\kappa +1}{\kappa} \right)^{n-1}, \quad \forall n\ge 1.
\]
Since $(P_n)_{n\ge 0}$ is a sequence of orthogonal polynomials with respect to the measure $\Pi_\kappa$, we have 
\[
\int_\R t^k P_n(t) d\Pi_\kappa(t) = 0 \quad \text{for $0 \le k < n$.}
\]
Therefore, 
\begin{align*}
\| P_n\|_{L^2(\Pi_\kappa)}^2 = \int_\R P_n(t)^2 d\Pi_\kappa(t)=  k_n  \int_\R t^n P_n(t) d\Pi_\kappa(t), \quad \forall n \ge 0.  
\end{align*}
Hence  for any $n \ge 1$, by using the recursion relation \eqref{cartier-dunau}, we obtain 
\begin{align*}
\int_\R t^n P_n(t) d\Pi_\kappa(t)  & =  \int_\R t^{n-1} \Big[ \frac{\kappa}{ \kappa + 1} P_{n+1} (t) + \frac{1}{\kappa+1} P_{n-1} (t)\Big] d\Pi_\kappa(t) 
\\
& =  \frac{1}{\kappa + 1}   \int_\R t^{n-1} P_{n-1}(t) d\Pi_\kappa(t). 
\end{align*}
By repeating the above argument, we obtain 
\[
\int_\R t^n P_n(x) d\Pi_\kappa(t) = \frac{1}{(\kappa +1)^{n}} \int_\R   P_0(t) d\Pi_\kappa(t)  = \frac{1}{(\kappa+1)^n}. 
\]
Therefore, we obtain the desired equality
\[
\| P_n\|_{L^2(\Pi_\kappa)}^2 = \frac{k_n}{(k+1)^n} =  
\left\{
\begin{array}{cc}
 1& \text{if $n =0$}
\vspace{2mm}
\\ 
\frac{1}{\kappa^{n-1} (\kappa +1)} & \text{if $n \ge 1$}
\end{array}
\right..
\]
\end{proof}

\begin{proof}[Proof of Lemma \ref{lem-cp-cont}]
We will prove Lemma \ref{lem-cp-cont} for $n \ge 2$, the proof in the case $n =0$ or $n =1$ is similar. 

Assume $n\ge 2$ and $\varphi_0\in L^\infty(\Pi_\kappa)$.  Write $\alpha = \widehat{\varphi_0} \in \mathcal{A}$. By  \eqref{bdd-assump}, we have 
\begin{align}\label{disc-l2-norm}
\sum_{l =0}^\infty   \kappa^l \cdot |\alpha(l)|^2<\infty. 
\end{align}
Using the explicit formula \eqref{comp-alpha-alpha} and the assumption that $n \ge 2$,  we have 
\begin{align*}
(\widehat{\varphi} \circledast_\kappa \widehat{\varphi_0}) (n)  = &   \sum_{l = 0}^\infty  \kappa^l[ \widehat{\varphi} (l)\alpha(l+n) + \widehat{\varphi}(l+n) \alpha(l)] + \sum_{i =1}^{n-1}    \widehat{\varphi} (i) \alpha(n-i)   
\\
& + \sum_{i =1}^{n-1}   \sum_{l =1}^\infty  (\kappa-1) \kappa^{l-1}   \widehat{\varphi} (l+i) \alpha( l +  (n-i)) 
\\
=&   \sum_{l=0}^\infty   \kappa^l \cdot  C(l, \kappa, \alpha, n) \cdot  \widehat{\varphi}(l),
\end{align*}
where $C(l, \kappa, \alpha, n)$ is a constant depends on $\alpha, \kappa, l, n$. It is easy to see that 
\[
|C(l, \kappa, \alpha, n)|\le  2 n \sum_{i \ge 0,\,   |i - l|\le n} |\alpha(i)|.
\]
Consequently, by  \eqref{disc-l2-norm}, we have 
\[
\sum_{l=0}^\infty   \kappa^l |C(l, \kappa, \alpha, n)|^2 <\infty.
\]
By Lemma \ref{lem-L2-norm}, there exists $c(\kappa)>0$, such that 
\[
\|P_{l}\|_{L^2(\Pi_\kappa)}^2 \le   c(\kappa) \kappa^{-l}, \quad \forall l \ge 0.
\]
Since $(P_n)_{n\ge 0}$ is an orthogonal sequence in $L^2(\Pi_\kappa)$, we have 
\begin{align*}
\Big\|\sum_{l =0}^\infty \kappa^l \cdot C(l, \kappa, \alpha, n)  P_l \Big\|_{L^2(\Pi_\kappa)}^2& = \sum_{l=0}^\infty \kappa^{2l} | C(l, \kappa, \alpha, n)|^2 \cdot \|P_{l}\|_{L^2(\Pi_\kappa)}^2 
\\
 & \le c(\kappa) \sum_{l=0}^\infty \kappa^{l} | C(l, \kappa, \alpha, n)|^2   <\infty.
\end{align*}
Therefore, the series 
\[
\psi(t): = \sum_{l =0}^\infty \kappa^l \cdot C(l, \kappa, \alpha, n)  P_l (t)
\]
defines a function in $L^2(\Pi_\kappa)$. Now we can write 
\[
(\widehat{\varphi} \circledast_\kappa \widehat{\varphi_0}) (n)  = \sum_{l =0}^\infty \kappa^l \cdot C(l, \kappa, \alpha, n) \int_\R P_l(t) \varphi(t) d\Pi_\kappa(t)   = \int_\R \psi(t) \varphi(t) d\Pi_\kappa(t).
\]
Since $\psi \in L^2(\Pi_\kappa) \subset L^1 (\Pi_\kappa)$,  the map 
\[
L^\infty(\Pi_\kappa) \ni  \varphi \mapsto (\widehat{\varphi} \circledast_\kappa \widehat{\varphi_0}) (n)   = \int_\R \psi  \varphi d\Pi_\kappa 
\]
is clearly continuous with respect to the weak-star topology on $L^\infty(\Pi_\kappa)$.  
\end{proof}

\section{The common invariant subspace}
\begin{proof}[Proof of Lemma \ref{lem-inv}]
Recall the definition \eqref{def-hat} of $\widehat{\varphi}$ for a function $\varphi \in L^1(\Pi_\kappa)$. By using the fact that $(P_n)_{n\ge 0}$ are orthogonal with respect to $\Pi_\kappa$  and using  the following equality obtained in Lemma \ref{lem-L2-norm}:
\begin{align}\label{sq-norm-sp}
\| P_n\|_{L^2(\Pi_\kappa)}^2 = \frac{1}{ \#S_n(o)},
\end{align}
 we have 
\[
\widehat{P_n}=  \frac{\delta_n}{\# S_n(o)} \in \mathcal{A}(\kappa), \quad \forall n \ge 0. 
\]
Recall the definition \eqref{onb-f-n} of the natural orthonormal basis $(f_n)_{n\ge 0}$ of the space $\ell^2(V_\kappa)_{\rad(o)}$.  Clearly, we have 
\begin{align}\label{f-n-Sn}
f_n(y)  = \frac{\mathds{1}_{S_n(o)}(y)}{\sqrt{\# S_n(o)}}=  \sqrt{\# S_n(o)}\cdot T[P_n] (y,o).
\end{align}
Therefore, for any $T[\varphi] \in B_\Gamma(\ell^2(V_\kappa))$, we have 
\begin{align}\label{T-varphi-f-n}
\begin{split}
(T[\varphi] f_n) (x) & = \sum_{h \in V_\kappa} T[\varphi](x, y) f_n(y) =   \sqrt{\# S_n(o)} \cdot \sum_{y \in V_\kappa} T[\varphi](x, y) \cdot T[P_n] (o, y) 
\\
& =   \sqrt{\# S_n(o)} \cdot  (T[\varphi] T[P_n])(x,o)  =  \sqrt{\# S_n(o)} \cdot(T[\varphi P_n])(x,o)
\\
&  = \sqrt{\# S_n(o)} \cdot \widehat{\varphi P_n} (d(x,o)),
\end{split}
\end{align}
where we used the equality $T[\varphi] T[P_n] = T[\varphi P_n]$. 
By definition of $\ell^2(V_\kappa)_{\rad(o)}$, we have 
\[
T[\varphi] f_n \in \ell^2(V_\kappa)_{\rad(o)}.
\] 
This shows that $\ell^2(V_\kappa)_{\rad(o)}$ is an invariant subspace for $T[\varphi]$ and it follows that it is a common invariant subspace of all operators in $B_\Gamma(\ell^2(V_\kappa))$. 
\end{proof}

\begin{proof}[Proof of Theorem \ref{prop-inv}]
An orthonormal basis of $L^2(\Pi_\kappa)$ is given by 
\[
\Big( \frac{P_n}{\| P_n\|_{L^2(\Pi_\kappa)}}\Big)_{n\ge 0}.
\]
Therefore, the operator $U_o: \ell^2(V_\kappa)_{\rad(o)} \rightarrow L^2(\Pi_\kappa)$ defined by \eqref{def-U-o} is a unitary operator.

 We now show that we do have the commutative diagram \eqref{CD-unitary}. Indeed, we only need to show that for any $\varphi \in L^\infty(\Pi_\kappa)$ and any $n\ge 0$, the equality 
\begin{align}\label{CD-equ-ev}
U_oT[\varphi]\left( \frac{\mathds{1}_{S_n(o)}}{\sqrt{\# S_n(o)}} \right)  =  M_\varphi U_o \left( \frac{\mathds{1}_{S_n(o)}}{\sqrt{\# S_n(o)}} \right)
\end{align}
holds. Using \eqref{T-varphi-f-n}, we have 
\[
T[\varphi]\left( \frac{\mathds{1}_{S_n(o)}}{\sqrt{\# S_n(o)}} \right)   =  \sqrt{\# S_n(o)} \cdot \sum_{k = 0}^\infty \widehat{\varphi P_n}  (k) \mathds{1}_{S_k(o)}.
\]
Hence, by recalling \eqref{sq-norm-sp}, we have
\begin{align*}
U_oT[\varphi]\left( \frac{\mathds{1}_{S_n(o)}}{\sqrt{\# S_n(o)}} \right)  & = \sqrt{\# S_n(o)} \cdot \sum_{k = 0}^\infty \widehat{\varphi P_n}  (k) \sqrt{\# S_k(o)}  U_o \left(  \frac{ \mathds{1}_{S_k(o)}}{\sqrt{\# S_k(o)}} \right)
\\
& =  \sqrt{\# S_n(o)} \cdot \sum_{k = 0}^\infty  \langle \varphi P_n, P_k \rangle_{L^2(\Pi_\kappa)} \sqrt{\# S_k(o)} \frac{P_k}{\| P_k\|_{L^2(\Pi_\kappa)}}
\\
& = \sqrt{\# S_n(o)} \cdot \sum_{k = 0}^\infty  \left \langle \varphi P_n, \frac{P_k}{\| P_k\|_{L^2(\Pi_\kappa)}}  \right\rangle_{L^2(\Pi_\kappa)} \frac{P_k}{\| P_k\|_{L^2(\Pi_\kappa)}}
\\
& = \sqrt{\# S_n(o)} \cdot \varphi P_n,
\end{align*}
where all the series are considered as $L^2(\Pi_\kappa)$-convergent series. On the other hand, 
\begin{align*}
 M_\varphi U_o \left( \frac{\mathds{1}_{S_n(o)}}{\sqrt{\# S_n(o)}} \right) = M_\varphi \left( \frac{P_n}{\| P_n\|_{L^2(\Pi_\kappa)}} \right) = \sqrt{\# S_n(o)} \cdot \varphi P_n.
\end{align*}
Thus we verified the equality \eqref{CD-equ-ev} for any $\varphi \in L^\infty(\Pi_\kappa)$ and any integer $n\ge 0$.
\end{proof}

\section{Appendix}
\subsection{The orthogonality measure of Cartier-Dunau polynomials }
 Here we use a  special case of  a result of Cohen-Trenholme \cite[Theorem 3]{Cohen84} on the calculation of the measure for which a sequence of polynomials with a constant recrusion formula is orthogonal. Recall the Cartier-Dunau polynomials defined in \eqref{cartier-dunau}.  Let $k_n$ be the leading coefficient of $P_n$. The recursion formula \eqref{cartier-dunau} implies that 
\begin{align}\label{leading-coef}
k_n =     \left(\frac{\kappa +1}{\kappa} \right)^{n-1}, \quad \forall n\ge 1.
\end{align}
Now  reduce $(P_n)_{n\ge 1}$ to a sequence $(Q_n)_{n\ge 1}$ of monic polynomials by
\[
Q_n(t) = \frac{P_n(t)}{k_n}  =\left( \frac{\kappa}{\kappa + 1}\right)^{n-1} P_n(t), \quad \forall n \ge 1
\]
and set $Q_0(t) \equiv \frac{\kappa+1}{\kappa}$, 
we obtain the recursion for the sequence $(Q_n)_{n\ge 0}$: 
\begin{align}\label{rec-rel}
Q_0(t) = \frac{\kappa + 1}{\kappa}, \quad Q_1(t) = t,  \quad Q_{n+1}(t) = tQ_n(t) - \frac{\kappa}{(\kappa + 1)^2} Q_{n-1}(t), \quad n \ge 1.
\end{align}
Then by applying \cite[Theorem 3]{Cohen84}, the measure $d\nu$ (unique up to a multiplicative constant) with respect to  which   $(Q_n)_{n\ge 0}$ are orthogonal polynomials is given by 
\[
d \nu(t) = \frac{\sqrt{4\kappa (\kappa + 1)^{-2} - t^2}}{1 - t^2}  \mathds{1}_{[- \frac{2 \sqrt{\kappa} }{ \kappa + 1}, \,  \frac{2 \sqrt{\kappa} }{ \kappa + 1} ]} (t) dt.
\]
By elementary calculus, we have 
\[
\int_{\R} d\nu = \frac{2\pi}{\kappa + 1}.
\]
It follows that the measure  $\Pi_\kappa$ given by the formula \eqref{pi-kappa}  is a probability measure and it is the unique one with respect to which $(P_n)_{n\ge 0}$ is a sequence of orthogonal polynomials.


\end{document}